\documentclass[11pt]{amsart}
\usepackage{amssymb}
\usepackage{amsmath}
\usepackage[all, cmtip, line]{xy}
\usepackage{tikz-cd}
\usepackage{pbox}
\usepackage{array}
\usepackage[utf8]{inputenc}
\usepackage{color}

\newtheorem{thm}{Theorem}[section]

\newtheorem{cor}[thm]{Corollary}
\newtheorem{lem}[thm]{Lemma}
\newtheorem{rem}[thm]{Remark}

\newcommand{\A}{{\mathcal A}}

\newcommand{\C}{{\mathcal C}}
\newcommand{\D}{{\mathcal D}}

\newcommand{\E}{{\mathcal E}}

\newcommand{\M}{{\mathcal M}}
\newcommand\Rep{\operatorname{Rep}}
\newcommand{\Irr}{\operatorname{Irr}}
\newcommand\FPdim{\operatorname{FPdim}}

\newcommand\vect{\operatorname{Vec}}
\newcommand\SuperV{\operatorname{SuperVec}}

\begin{document}
\title[Integral almost square-free modular categories]{Integral almost square-free modular categories}
\author[Dong]{Jingcheng Dong}
\email[Dong]{dongjc@njau.edu.cn}
\address[Dong]{College of Engineering, Nanjing Agricultural University, Nanjing, Jiangsu 210031, China}

\author[Li]{Libin Li}
\email[Li]{lbli@yzu.edu.cn}
\address[Li]{Department of Mathematics, Yangzhou University, Yangzhou, Jiangsu 225002, China}

\author[Dai]{Li Dai}
\email[Dai]{daili1980@njau.edu.cn}
\address[Dai]{College of Engineering, Nanjing Agricultural University, Nanjing, Jiangsu 210031, China}

\keywords{modular category; group-theoretical fusion category; solvable fusion category; equivariantization; Frobenius-Perron dimension}

\subjclass[2010]{18D10; 16T05}

\date{\today}


\begin{abstract}
We study integral almost square-free modular categories; i.e., integral modular categories of Frobenius-Perron dimension $p^nm$, where $p$ is a prime number, $m$ is a square-free natural number and ${\rm gcd}(p,m)=1$. We prove that if $n\leq 5$ or $m$ is prime with $m<p$ then they are group-theoretical. This generalizes several results in the literature and gives a partial answer to the question posed by the first author and H. Tucker. As an application, we prove that an integral modular category whose Frobenius-Perron dimensions is odd and less than $1125$ is group-theoretical.
\end{abstract}

 \maketitle



\section{Introduction}\label{sec1}
Modular categories arise from many areas of mathematics and physics, including representation theory of quantum groups \cite{BaKi2001lecture}, vertex operator algebras \cite{Huang2005vertex}, von Neumann algebras \cite{EvKa1998quantum}, topological quantum field theory \cite{Turaev1994quantum} and conformal field theory \cite{MoSe1989Classical}. The importance of modular categories also comes from their applications in condensed matter physics and quantum computing \cite{Wang2010Top}.

A fusion category is called \textbf{almost square-free (ASF)} if its \textbf{Frobenius-Perron (FP)} dimension has the form $p^nm$, where $p$ is a prime number, $m$ is a square-free natural number and ${\rm gcd}(p,m)=1$. Several examples of this class of fusion categories have been extensively studied in the literature, see \cite{bruillard2013classification, DoTu2015, dong2013semisimple, etingof2005fusion, naidu2011finiteness}.

One important class of fusion categories is that of group-theoretical fusion categories. As described by Etingof, Nikshych and Ostrik \cite{etingof2005fusion}, a group-theoretical fusion category can be explicitly constructed from finite group data (see Section \ref{sec25}). Therefore this class of fusion categories can be viewed as the simplest fusion categories except the pointed ones. This is similar in spirit to the work of classifying finite-dimensional Hopf algebras: if a finite-dimensional Hopf algebra can be obtained from group algebras or dual group algebras then it must be well-understood. The main task of this paper is to determine when an integral ASF modular category is group-theoretical.

The paper is organized as follows. In Section \ref{sec2}, we recall some basic definitions and results which will be used throughout.

In Section \ref{sec3}, we study the existence of non-trivial Tannakian subcategories and their FP dimensions in an integral ASF modular category. As one of the consequences, we get that any integral ASF modular categories with FP dimension $p^nm$ is group-theoretical if $n\leq 5$. This generalizes the work in a series of papers \cite{bruillard2013classification, DoTu2015, naidu2011finiteness}.


In Section \ref{sec5}, we study the core of an integral ASF modular category. This is a new notion introduced in \cite{drinfeld2010braided}. We prove that the core of an integral ASF modular category is a pointed modular category. As a consequence, we get the structure of an integral ASF modular category: it is an equivariantization of a nilpotent fusion category of nilpotency class $2$. As an application, we get that an integral modular category with FP dimension $p^nq$, where $q<p$ are prime numbers, is group-theoretical. This gives a partial answer to the question posed in \cite{DoTu2015}.

In Section \ref{sec6}, we apply the results obtained so far to a low-dimensional integral modular category $\C$. We find that if the FP dimension $\FPdim(\C)$ is odd and less than 1125 then $\C$ is ASF except the case when $\FPdim(\C)=675$. We prove that if $\FPdim(\C)=675$ then $\C$ is a pointed modular category. This allows us to conclude that integral modular categories whose Frobenius-Perron dimensions are odd and less than $1125$ are group-theoretical.

\section{Preliminaries}\label{sec2}
We shall work over an algebraically closed field $k$ of characteristic zero. A fusion category is an abelian $k$-linear semisimple rigid monoidal category with a simple unit object $\textbf{1}$, finite-dimensional morphisms space, and finitely many isomorphism classes of simple objects. We refer the reader to \cite{drinfeld2010braided,etingof2005fusion} for the main notions about fusion categories. For the reader's convenience, we collect some definitions and basic results in this section.

\subsection{Frobenius-Perron dimension}\label{sec21}
Let $\C$ be a fusion category, and let $\Irr(\C)$ denote the set of isomorphism classes of simple objects of $\C$. The FP dimension of a simple object $X\in\C$ is the FP eigenvalue of the matrix of left multiplication by the class of $X$ in $\Irr(\C)$. The FP dimension of $\C$ is the number $$\FPdim(\C)=\sum_{X\in\Irr(\C)}\FPdim(X)^2.$$

A fusion category $\C$ is called integral if $\FPdim(X)$ is an integer for all simple object $X\in\C$, and it is called weakly integral if $\FPdim(\C)$ is an integer.

Let $\Irr_\alpha(\C)$ be the set of isomorphism classes of simple objects of FP dimension $\alpha$. Then the set $\Irr_1(\C)$ is the set of all invertible simple objects, and it generates the unique largest pointed fusion subcategory $\C_{pt}$ of $\C$. Recall that a fusion category is called pointed if every simple object has Frobenius-Perron dimension $1$. The set $\Irr_1(\C)$ also forms a group with multiplication given by tensor product. We denote this group by $G(\C)$. This group admits an action on the set $\Irr(\C)$ by left tensor product. For $X\in \Irr(\C)$, we use $G[X]$ to denote the stabilizer of $X$ under this action.

\begin{lem}\label{lem21}
Let $\C$ be a fusion category. Suppose that the order of $G(\C)$ is a power of a prime number $p$ and ${\rm gcd}(|\Irr_{\alpha}(\C)|, p)=1$ for some $\alpha>1$. Then there exists $X\in\Irr_\alpha(\C)$ such that $G[X]=G(\C)$.
\end{lem}

\begin{proof}
Obviously, the action of $G(\C)$ on the set $\Irr(\C)$ preserves FP dimension. Hence, this action can be restricted to the set $\Irr_\alpha(\C)$. The set $\Irr_\alpha(\C)$ is therefore a union of disjoint orbits under this restricted action, and every orbit has length $p^i$ for some $i\geqslant0$. Since ${\rm gcd}(|\Irr_{\alpha}(\C)|, p)=1$ there exists at least one orbit having length 1, which implies the lemma.
\end{proof}

\subsection{Group extensions}\label{sec22}
Let $G$ be a finite group and let $e\in G$ be the identity element. A fusion category $\C$ is said to have a $G$-grading if $\C$ has a direct sum of full abelian subcategories $\C=\oplus_{g\in G}\C_g$ such that $(\C_g)^\ast=\C_{g^{-1}}$ and $\C_g\otimes\C_h\subseteq\C_{gh}$ for all $g,h\in G$. The grading $\C=\oplus_{g\in G}\C_g$ is called faithful if $\C_g\neq 0$ for all $g\in G$. A fusion category $\C$ is said a $G$-extension of $\D$ if $\C$ admits a faithful grading $\C=\oplus_{g\in G}\C_g$ such that $\C_e\cong\D$.

By \cite[Proposition 8.20]{etingof2005fusion}, if $\C$ is a $G$-extension of $\D$ then, for all $g,h\in G$, we have
\begin{equation}\label{eq0}
\begin{split}
\FPdim(\C_g)=\FPdim(\C_h)\, \mbox{and}\, \FPdim(\C)=|G| \FPdim(\D).
\end{split}
\end{equation}

It is known that every fusion category $\C$ has a canonical faithful grading $\C=\oplus_{g\in \mathcal{U}(\C)}\C_g$ with trivial component $\C_e=\C_{ad}$, where $\mathcal{U}(\C)$ is the universal grading group of $\C$, and $\C_{ad}$ is the adjoint subcategory of $\C$ generated by simple objects in $X\otimes X^\ast$ for all $X\in \Irr(\C)$. This grading is called the universal grading of $\C$.

\subsection{Equivariantizations and de-equivariantizations}\label{sec23}
Let $\C$ be a fusion category with an action of a finite group $G$. We then can define a new fusion category $\C^G$ of $G$-equivariant objects in $\C$. An object of this category is a pair $(X,(u_g)_{g\in G})$, where $X$ is an object of $\C$, $u_g : g(X)\to X$ is an isomorphism for all $g\in G$, such that
$$u_{gh}\circ \alpha_{g,h} =u_g \circ g(u_h),$$
where $\alpha_{g,h}: g(h(X))\to gh(X)$ is the natural isomorphism associated to the action. Morphisms
and tensor product of equivariant objects are defined in an obvious way. This new category is called the $G$-equivariantization of $\C$.

In the other direction, let $\C$ be a fusion category and let
$\E = \Rep(G)\subseteq \mathcal{Z}(\C)$ be a Tannakian subcategory that embeds into $\C$ via the forgetful
functor $\mathcal{Z}(\C)\to \C$. Here $\mathcal{Z}(\C)$ denotes the Drinfeld center of $\C$. Let $A={\rm Fun}(G)$ be the algebra of functions on $G$. It is a commutative algebra in $\mathcal{Z}(\C)$. Let $\C_G$ denote the category of left $A$-modules in $\C$. It is a fusion category and called the de-equivariantization of $\C$ by $\E$. See \cite{drinfeld2010braided} for details on equivariantizations and de-equivariantizations.

Equivariantizations and de-equivariantizations are inverse to each other:
$$(\C_G)^G\cong\C\cong(\C^G)_G,$$
and their FP dimensions have the following relations:
\begin{equation}\label{eq1}
\begin{split}
\FPdim(\C)=|G|\FPdim(\C_G) \,\mbox{\,and}\, \FPdim(\C^G)=|G| \FPdim(\C).
\end{split}
\end{equation}

\subsection{Solvable fusion categories}\label{sec24}
A fusion category $\C$ is nilpotent if there exists a sequence of fusion categories $\vect=\C_0\subseteq\C_1\subseteq\cdots\subseteq\C_n=\C$, and a sequence of finite groups $G_1,\cdots,G_n$ such that $\C_i$ is a $G_i$-extension of $\C_{i-1}$ for all $i$. If the groups $G_1,\cdots,G_n$ are cyclic groups of prime order then $\C$ is called cyclically nilpotent.

 A fusion category $\C$ is called solvable if it is Morita equivalent to a cyclically nilpotent fusion category. Here the notion of Morita equivalence is a categorical analogue of the notion of Morita equivalence for rings. Precisely, two fusion categories $\C$ and $\D$ are Morita equivalent if $D$ is equivalent to the dual $\C^\ast_{\M}$ with respect to an indecomposable module category $\M$, where $\C^\ast_{\M}$ is the category of $\C$-module endofunctors of $\M$.

Solvable fusion categories have good properties: If $\C$ is a solvable fusion category then $\C$ can be obtained by recursive extensions or equivariantizations from the trivial fusion category $\vect$, where all finite groups involved are cyclic groups of prime order. It is shown in \cite[Proposition 4.5]{etingof2011weakly} that the class of solvable fusion categories is closed under taking extensions and equivariantizations by solvable groups.

\subsection{Group-theoretical fusion categories}\label{sec25}
A fusion category $\C$ is called group-theoretical if it is Morita equivalent to a pointed fusion category. A group-theoretical fusion category can be precisely reconstructed from finite group data as follows:

Let $H$ be a subgroup of a finite group $G$, $\omega \in Z^3(G, \mathbb{C}^{\times})$ a normalized 3-cocycle, and $\psi \in C^2(H, \mathbb{C}^{\times})$ a normalized 2-cochain such that $\mathrm{d}\psi = \omega |_H$. Let $\vect_G^\omega$ be the category of $G$-graded vector spaces with associativity given by 3-cocycle $\omega$. The twisted group algebra $\mathbb{C}_\psi[H]$ is an associative algebra in $\vect_G^\omega$. Therefore we may consider the category
\begin{equation*}
\C(G, \omega, H, \psi) := \{\mathbb{C}_\psi[H]-\text{bimodules in} \vect_G^\omega\}
\end{equation*}
with tensor product $\otimes_{\mathbb{C}_\psi[H]}$ and unit object $\mathbb{C}_\psi[H]$. It is a group-theoretical fusion category. In fact every group-theoretical fusion category can be obtained in this way  \cite[Section 8.8]{etingof2005fusion}. Note that a group-theoretical fusion category must be integral \cite[Corollary 8.43]{etingof2005fusion}.

\subsection{Braided fusion categories}\label{sec26}
A fusion category $\C$ is called braided if it admits a braiding $c$, where the braiding $c$ is a family of natural isomorphisms: $c_{X,Y}$:$X\otimes Y\rightarrow Y\otimes X$ satisfying the hexagon axioms for all $X,Y\in\C$ \cite{kassel1995quantum}.

A braided fusion category $\C$ is called symmetric if $c_{Y,X}c_{X,Y}=id_{X\otimes Y}$ for all objects $X,Y\in\C$. A symmetric fusion category $\C$ is said to be Tannakian if it is equivalent to $\Rep(G)$ for some finite group $G$ as symmetric categories.

Let $G$ be a finite group and let $u\in G$ be a central element of order 2. Then the category $\Rep(G)$ has a braiding $c^u_{X,Y}$ as follows: for all\, $x\in X, y\in Y$,
\begin{equation}\label{eq2}
\begin{split}
c^u_{X,Y}(x\otimes y)=(-1)^{mn}y\otimes x\,\, \mbox{if}\,\,ux=(-1)^mx,uy=(-1)^ny.
\end{split}\nonumber
\end{equation}

Let $\Rep(G,u)$ be the fusion category $\Rep(G)$ equipped with the new braiding $c^u_{X,Y}$. Deligne proved that any symmetric fusion category is equivalent to some $\Rep(G,u)$ \cite{deligne1990categories}. The following lemma is taken from \cite[Corollary 2.50]{drinfeld2010braided}.

\begin{lem}\label{lem22}
Let $\C$ be the symmetric fusion category $\Rep(G,u)$. Then one of the following holds:

(1) $\C$ is a Tannakian category;

(2) $\Rep(G/\langle u\rangle)$ is a Tannakian subcategory of $\C$ with dimension $\frac{1}{2}\FPdim(\C)$. In particular, if $\FPdim(\C)$ is odd then $\C$ is Tannakian.
\end{lem}

Note that the lemma above shows that if $\FPdim(\C)$ is bigger than 2, then $\C$ always has a non-trivial Tannakian subcategory $\Rep(G/\langle u\rangle)$.

\medbreak
Let $\D\subset\C$ be a fusion subcategory. Then the M\"{u}ger centralizer $\D'$ of $\D$ in $\C$ is the fusion subcategory $$\D'=\{Y\in\C|c_{Y,X}c_{X,Y}=id_{X\otimes Y}\, \mbox{for all}\, X\in\D\}.$$
The M\"{u}ger center $\mathcal{Z}_2(\C)$ of $\C$ is the M\"{u}ger centralizer $\C'$ of $\C$. The fusion category $\C$ is called non-degenerate if its M\"{u}ger center $\mathcal{Z}_2(\C)$ is trivial.

A braided fusion category is called premodular if it admits a spherical structure. A modular category is a non-degenerate premodular category. Combining Proposition 8.23 with Proposition 8.24 of \cite{etingof2005fusion}, we know that a weakly integral braided fusion category is modular if and only if it is non-degenerate. Thus, an ASF fusion category is modular if and only if it is non-degenerate.

Suppose that $\C$ is a modular category and $\D \subseteq \C$ is a fusion subcategory. The following equalities will be frequently used in this paper:
\begin{align}
(\C_{pt})'& = \C_{ad}\label{cptad};\\
\FPdim(\D) \FPdim(\D') &= \FPdim(\C)\label{fpdim}.
\end{align}
The first equality comes from \cite[Corollary 6.8]{gelaki2008nilpotent}, and the second one comes from \cite[Theorem 3.2]{muger2003structure}.

Let $\E=\Rep(G)$ be a Tannakian category, and let $\C$ be a braided fusion category containing $\Rep(G)$. Recall from \cite[Section 4.4.3] {drinfeld2010braided} that the de-equivariantization $\C_G$ of $\C$ by $\E$ is a braided $G$-crossed fusion category. Let $\C_G=\oplus_{g\in G}(\C_G)_g$ be the corresponding grading. The trivial component $(\C_G)_e$ of this grading is a braided fusion category. By \cite[Proposition 4.56]{drinfeld2010braided}, the braided fusion category $\C$ is non-degenerate if and only if the trivial component $(\C_G)_e$ is non-degenerate and the corresponding grading of $\C_G$ is faithful. This description implies the following lemma.

\begin{lem}\label{newadded}
If a modular  category $\C$ contains a Tannakian subcategory $\Rep(G)$, then the square of $|G|$ divides $\FPdim(\C)$. In particular, if $\C$ is an ASF modular category with FP dimension $p^nm$, then the FP dimension of any Tannakian subcategory of $\C$ is a power of $p$.
\end{lem}

Let $\SuperV$ be the category of super vector spaces. We recall a helpful lemma of M\"uger's regarding braided fusion categories containing $\SuperV$.

\begin{lem}\label{mugerlem}\cite[Lemma 5.4]{muger2000galois}
Let $\C$ be a braided fusion category such that $\SuperV \subseteq \mathcal{Z}_2(\C)$ and let $g \in \C$ be the invertible object generating $\SuperV$. Then $g \otimes X \ncong X$ for all $X \in \Irr(\C)$.
\end{lem}

\section{Group-theoretical properties of integral ASF modular categories}\label{sec3}
\setcounter{equation}{0}
In this section, we aim to study the group-theoretical property of an integral ASF modular category, so we may assume all fusion categories considered are integral.

Throughout this section, $\C$ is an integral modular category of FP dimension $p^nm$, where $p$ is a prime number, $m$ is a square-free natural number and ${\rm gcd}(p,m)=1$. We may assume that $m>1$ since \cite[Corollary 6.8]{drinfeld2007group} shows that an integral fusion category of prime power dimension is always group-theoretical.

Note that an integral modular category $\C$ of square-free FP dimension must be pointed.  In fact, \cite[Lemma 1.2]{etingof1998some} shows that the square of $\FPdim(X)$ divides $\FPdim(\C)$ for every simple object $X$ of $\C$. Since $\FPdim(\C)$ is square-free, $\FPdim(X)$ must be $1$. Hence $\C$ is pointed.  We therefore will always assume in the following context that $n\geq2$.

Let $n=2t+1$ if $n$ is odd, or $n=2t$ if $n$ is even. Then \cite[Lemma 1.2]{etingof1998some} shows that the possible FP dimensions of simple objects of $\C$ are $1, p, \cdots, p^t$. Let $a_0, a_1, \cdots, a_t$ be the number of non-isomorphic simple objects of $\C$ of FP dimension $1, p, \cdots, p^t$, respectively. Then we get an equation:
\begin{equation}\label{eq3}
\begin{split}
a_0+a_1p^2+\cdots+a_tp^{2t}=p^nm.
\end{split}
\end{equation}

 Since we have assumed that $n\geq2$, we get that $p^2$ divides $a_0=\FPdim(\C_{pt})$  by equation (\ref{eq3}). This observation implies the following result.

\begin{lem}\label{lem30}
The FP dimension of $\C_{pt}$ is divisible by  $p^2$.
\end{lem}
Since $\C$ is modular, the universal group $\mathcal{U}(\C)$ is isomorphic to the group $G(\C)$ consisting of the isomorphism classes of invertible simple objects \cite[Theorem 6.2]{gelaki2008nilpotent}. Hence $|\mathcal{U}(\C)|=\FPdim(\C_{pt})$ is divisible by $p^2$.

\begin{lem}\label{lem31}
Suppose that $\C$ is not pointed. Then $\FPdim(\C_{pt})$ is not divisible by $p^{n-1}$.
\end{lem}

\begin{proof}
Suppose first that $\FPdim(\C_{pt})=p^nm'$ for some $m'\in \mathbb{N}$. By equation (\ref{eq0}), we have
$$|\mathcal{U}(\C)|\FPdim(\C_{ad})=\FPdim(\C_{pt})\FPdim(\C_{ad})=\FPdim(\C).$$
This implies that $\FPdim(\C_{ad})=m''>1$ is square-free and is not divisible by $p$, where $m'm''=m$.

Since $\C_{ad}$ is braided, the FP dimension of every simple object of $\C_{ad}$ divides $\FPdim(\C_{ad})$ \cite[Theorem 2.11]{etingof2011weakly}. Hence $\C_{ad}$ does not contain simple objects of FP dimension $p^i$, for any $1\leqslant i\leqslant t$. That is, $\C_{ad}$ is pointed. The subcategory $\C_{pt}$ is the unique largest pointed fusion subcategory of $\C$. Hence $\C_{ad}$ is a fusion subcategory of $\C_{pt}$. It follows that $\FPdim(\C_{ad})$ divides $\FPdim(\C_{pt})$ \cite[Proposition 8.15]{etingof2005fusion}. This is a contradiction since $1\neq \FPdim(\C_{ad})$ is relatively prime to $\FPdim(\C_{pt})$.

Suppose then that $\FPdim(\C_{pt})=p^{n-1}m'$ for some $m'\in \mathbb{N}$ with ${\rm gcd}(p,m')=1$. In this case $\FPdim(\C_g)=pm''$ for every component $\C_g$ of the universal grading of $\C$, for some $m''\in \mathbb{N}$. Let $a_g^0, a_g^1, \cdots, a_g^t$ be the number of non-isomorphic simple objects of $\C_g$ of dimension $1, p, \cdots, p^t$. Then we have an equation:
\begin{equation}\label{eq4}
\begin{split}
a_g^0+a_g^1p^2+\cdots+a_g^tp^{2t}=pm''.
\end{split}
\end{equation}
This equation implies that $a_g^0\neq 0$ and $p$ divides $a_g^0$. This further means that every component $\C_g$ contains at least $p$ non-isomorphic invertible simple objects. There are $p^{n-1}m'$ components in the universal grading, since $|\mathcal{U}(\C)|=\FPdim(\C_{pt})$. Therefore, there are at least $p^nm'$ non-isomorphic invertible simple objects in $\C$. This contradicts the fact $\FPdim(\C_{pt})=p^{n-1}m'$.
\end{proof}

\begin{cor}\label{cor32}
If $n\leqslant3$ then $\C$ is pointed.
\end{cor}

\begin{proof}
If $n=0$ or $1$ then $\FPdim(\C)$ is square-free. This case has been stated at the beginning of this section.

When $n=2$ or $3$. Suppose that $\C$ is not pointed. By Lemma \ref{lem30}, $\FPdim(\C_{pt})$ is divisible by $p^2$. But Lemma \ref{lem31} shows that it is impossible. Hence $\C$ is pointed.
\end{proof}

In view of Corollary \ref{cor32}, we assume that $n\geqslant4$ in the following context.

\begin{lem}\label{lem33}
The largest pointed fusion category $(\C_{ad})_{pt}$ of $\C_{ad}$ is a symmetric subcategory with FP dimension $p^j$ for some $2\leqslant j\leqslant n-2$. In particular, $\C$ has a non-trivial pointed Tannakian subcategory of dimension $p$.
\end{lem}
\begin{proof}
By Lemma \ref{lem30} and Lemma \ref{lem31}, we may write $\FPdim(\C_{pt})=p^im'$ for some $2\leq i\leq n-2$ and ${\rm gcd}(p,m')=1$. So we can write $\FPdim(\C_{ad})=p^{n-i}m''$, where $m'm''=m$.

Let $\D=(\C_{ad})_{pt}$. Then $\D$ is a fusion subcategory of $\C_{ad}$, as well as of $\C_{pt}$. Hence, dimension of $\D$ divides both $\FPdim(\C_{pt})$ and $\FPdim(\C_{ad})$. This implies that $\FPdim(\D)$ is a power of $p$. Let $a_e^0, a_e^1, \cdots, a_e^t$ be the number of non-isomorphic simple objects of $\C_e=\C_{ad}$ of dimension $1, p, \cdots, p^t$. Then we have an equation:
\begin{equation}\label{eq44}
\begin{split}
a_e^0+a_e^1p^2+\cdots+a_e^tp^{2t}=p^{n-i}m''.
\end{split}
\end{equation}
The fact $n-i\geq 2$ and the equation above imply that $\FPdim(\D)=a_e^0$ is divisible by $p^2$. So we may write $\FPdim(\D)=p^j$ for some $2\leq j\leq n-2$.

From $\D\subseteq\C_{ad}$ we have $(\C_{ad})'\subseteq\D'$. On the other hand $(\C_{ad})'=\C_{pt}\supseteq\D$ by \cite[Corollary 6.8]{gelaki2008nilpotent}. Hence, $\D\subseteq \D'$, and hence $\D$ is a symmetric fusion subcategory of $\C$.

If $p>2$ then $\FPdim(\D)$ is odd. In this case $\D\cong\Rep(G)$ is a Tannakian subcategory of $\C$. If $p=2$ then there exists a group $G$ of order $2^j$ and a central element $u\in G$ of order 2 such that $\D\cong \Rep(G,u)$ as symmetric fusion categories. It follows that $\Rep(G/\langle u\rangle)$ is a Tannakian subcategory of $\C$ with $\FPdim(\Rep(G/\langle u\rangle))=2^{j-1}$. In both cases, $G$ is abelian (since $\D$ is pointed), hence we get that $\D$, and hence $\C$ has a Tannakian subcategory of FP dimension $p$.
\end{proof}

\begin{thm}\label{thm35}
The modular category $\C$ is either group-theoretical or contains a Tannakian subcategory of FP dimension $p^i$, for some $i\geq2$.
\end{thm}

\begin{proof}
We keep all notations from  Lemma \ref{lem33} and let $\D=(\C_{ad})_{pt}$. If $p$ is odd then $\FPdim(\D)=p^j$ is also odd, and hence $\D$ is a Tannakian subcategory by Lemma \ref{lem22}. So we are done. In the rest of our proof, we assume that $p=2$.

Suppose that $\FPdim(\D)\geqslant2^3$. Then $\Rep(G/\langle u\rangle)$ is a Tannakian subcategory of $\C$ with $\FPdim(\Rep(G/\langle u\rangle))\geqslant2^2$, also by Lemma \ref{lem22} .

Suppose that $\FPdim(\D)=2^2$. Let $a_e^1, \cdots, a_e^t$ be the number of non-isomorphic simple objects of $\C_{ad}$ of dimension $2,2^2,\cdots, 2^t$. Then we have
\begin{equation}\label{eq5}
\begin{split}
4+\sum_{k=1}^ta_e^k2^{2k}=2^{n-i}m''.
\end{split}
\end{equation}
So we get
\begin{equation}\label{eq6}
\begin{split}
a_e^1=2^{n-i-2}m''-1-2\sum_{k=2}^ta_e^k2^{2k-3}.
\end{split}
\end{equation}

If $i<n-2$ then $a_e^1$ is odd. In this case, $|G(\C_{ad})|=\FPdim(\D)=4$ and $|\Irr_2(\C_{ad})|=a_e^1$ is odd. By Lemma \ref{lem21}, there exists $X\in \Irr_2(\C_{ad})$ such that $G[X]=G(\C_{ad})$. In other words, $h\otimes X\cong X$ for all $h\in G(\C_{ad})$.

If $\D$ is not Tannakian then it contains the category $\E$ of super vector spaces. Let $1\neq g\in \E$ be the unique invertible object which generates $\E$ as a symmetric category. Then $g\otimes X\ncong X$ for every simple object of $\C_{ad}$ by \cite[Lemma 5.4]{muger2000galois}. This contradicts the result obtained above. So in this case $\D$ must be Tannakian.

Now we consider the case $i=n-2$ and assume that $\D$ is not Tannakian. The argument in this case is pointed out to the author by Sonia Natale. In this case, $\FPdim(\D)=4$ and $\FPdim(\C_{ad})=4m''$. By Lemma \ref{lem33}, $\D$ has a Tannakian subcategory equivalent to $\Rep(\mathbb{Z}_2)$. This Tannakian subcategory is also contained in $\C_{ad}$. So we have the de-equivariantization $(\C_{ad})_{\mathbb{Z}_2}$ of $\C_{ad}$ by $\Rep(\mathbb{Z}_2)$. Since $(\C_{ad})'\supseteq\D\supseteq \Rep(\mathbb{Z}_2)$, $\Rep(\mathbb{Z}_2)$ is contained in the M\"{u}ger center of $\C_{ad}$. Hence the de-equivariantization $(\C_{ad})_{\mathbb{Z}_2}$ is braided by \cite[Lemma 3.10]{muger2000galois}. The canonical functor $F:\C_{ad}\rightarrow(\C_{ad})_{\mathbb{Z}_2}$ is a dominant braided tensor functor such that $F(\D)\cong\D_{\mathbb{Z}_2}$ by \cite[Proposition 4.22]{drinfeld2010braided}.

Consider the de-equivariantization $\D_{\mathbb{Z}_2}$. Since $\D$ is not Tannakian by our assumption, we have that $\D_{\mathbb{Z}_2}\cong {\rm SuperVec}$ as braided fusion categories \cite[Remark 9.1]{natale2014graphs}. Applying the functor $F$ to $\mathcal{Z}_2(\C_{ad})\supseteq\D$, we have
\begin{equation}\label{eq7}
\begin{split}
{\rm SuperVec}\cong \D_{\mathbb{Z}_2}\subseteq \mathcal{\mathcal{Z}}_2((\C_{ad})_{\mathbb{Z}_2}).
\end{split}
\end{equation}

By \cite[Lemma 7.2]{natale2012fusion}, the FP dimensions of simple objects of $(\C_{ad})_{\mathbb{Z}_2}$ are powers of 2. On the other hand, $(\C_{ad})_{\mathbb{Z}_2}$ is braided and hence the FP dimension of every simple object of it divides $\FPdim((\C_{ad})_{\mathbb{Z}_2})=2m''$. It follows that the category $(\C_{ad})_{\mathbb{Z}_2}$ only has simple objects with FP dimension 1 or 2. If $(\C_{ad})_{\mathbb{Z}_2}$ has a simple object $X$ of FP dimension 2 then equation (\ref{eq7}) implies that $g\otimes X\cong X$, where $g$ is the generator of $\SuperV$. This contradicts Lemma \ref{mugerlem}. Therefore, $(\C_{ad})_{\mathbb{Z}_2}$ is pointed.

We now consider the functor $\tilde{F}:\C\rightarrow\C_{\mathbb{Z}_2}$. Again by \cite[Proposition 4.22]{drinfeld2010braided}, we have
\begin{equation}\label{eq8}
\begin{split}
(\C_{\mathbb{Z}_2})_{ad}\subseteq\tilde{F}(\C_{ad})=(\C_{ad})_{\mathbb{Z}_2}.
\end{split}
\end{equation}

Hence $(\C_{\mathbb{Z}_2})_{ad}$ is also pointed. It follows that $\C_{\mathbb{Z}_2}$ is nilpotent, and hence $\FPdim(X)^2$ divides $\FPdim((\C_{\mathbb{Z}_2})_{ad})$ for all $X\in \Irr(\C_{\mathbb{Z}_2})$, by \cite[Corollary 5.3]{gelaki2008nilpotent}. Since $\FPdim(\C_{ad})_{\mathbb{Z}_2}=2m''$ is square-free, equation (\ref{eq8}) shows that $\FPdim((\C_{\mathbb{Z}_2})_{ad})$ is also square-free. Hence $\FPdim(X)=1$ for all $X\in \Irr(\C_{\mathbb{Z}_2})$. That is, $\C_{\mathbb{Z}_2}$ is pointed, hence $\C$ is group-theoretical by \cite[Theorem 7.2]{naidu2009fusion}. This completes the proof.
\end{proof}

\begin{cor}\label{cor35}
If $n\leq 5$ then $\C$ is group-theoretical.
\end{cor}
\begin{proof}
By Corollary \ref{cor32} and Theorem \ref{thm35}, it suffices to consider the case that $n=4$ or $5$, and $\C$ has a Tannakian subcategory $\Rep(G)$ of FP dimension $p^2$.

Let $\C_G$ be the de-equivariantization of $\C$ by $\Rep(G)$. Let $\C_G=\oplus_{g\in G}(\C_G)_g$ be the corresponding  grading of $\C_G$. This grading is faithful and $(\C_G)_e$ is a modular category (see Section \ref{sec26}). If $n=4$ then $\FPdim((\C_G)_e)=m$; if $n=5$ then $\FPdim((\C_G)_e)=pm$. In both cases, $(\C_G)_e$ is a pointed modular category by Corollary \ref{cor32}. Hence $\C_G$ is a nilpotent fusion category. By \cite[Corollary 5.3]{gelaki2008nilpotent}, $\FPdim(X)^2$ divides $\FPdim((\C_G)_e)$ for every simple object $X$ of $\C_G$. Since in both cases $\FPdim((\C_G)_e)$ is square-free, $\FPdim(X)=1$ for all simple object $X$ of $\C_G$. This shows that $\C_G$ is a pointed fusion category. Therefore, the modular category $\C$, being an equivariantization of a pointed fusion category, is group-theoretical by \cite[Theorem 7.2]{naidu2009fusion}.
\end{proof}

\section{General properties of integral ASF modular categories}\label{sec5}
\setcounter{equation}{0}
Let $\D$ be a braided fusion category. A Tannakian subcategory $\E\subset \D$ is maximal if it is not contained in any other Tannakian subcategory of $\D$. In this section, we still assume that $\C$ is an ASF modular category with FP dimension $p^nm$ as in Section \ref{sec3}.

\begin{thm}\label{thm34}
Let $\E\cong\Rep(G)$ be a proper maximal Tannakian subcategory of $\C$, and let $\E'$ be its M\"{u}ger centralizer in $\C$. Then the de-equivariantization $(\E')_G$ of $\E'$ by $\Rep(G)$ is pointed. In particular, $\E'$ is group-theoretical.
\end{thm}

The existence of $\E$ is guaranteed by Lemma \ref{lem33} and the modularity of $\C$.

\begin{proof}
Let $\D=(\E')_G$. The Tannakian subcategory $\E$ is the M\"{u}ger center of $\E'$. By \cite[Remark 2.3]{etingof2011weakly}, the de-equivariantization $\D$ is non-degenerate. Also, the FP dimension of $\E$ is a power of $p$ by Lemma \ref{newadded}. So $\D$ is an integral ASF modular category. Hence, if $\D$ is not pointed then $\D_{pt}\cap(\D_{pt})'=\D_{pt}\cap\D_{ad}=(\D_{ad})_{pt}$ has FP dimension $p^j$ for some $2\leq j\leq n-2$,  by Lemma \ref{lem33}.

On the other hand, the fusion category $\D$ is called the core of $\C$ in \cite[Section 5.4]{drinfeld2010braided}. It is a weakly anisotropic braided fusion category by \cite[Corollary 5.19]{drinfeld2010braided}. Recall that a braided fusion category is called weakly anisotropic if it has no nontrivial Tannakian subcategories which are stable under all braided autoequivalences of the fusion category. By \cite[Corollary 5.29]{drinfeld2010braided}, $\D_{pt}\cap(\D_{pt})'$ is either trivial, or equivalent to the category ${\rm SuperVec}$ of super vector spaces. Therefore, $\D$ must be pointed, otherwise we will get a contradiction.

The last statement follows from \cite[Theorem 7.2]{naidu2009fusion} which says that a braided fusion category is group-theoretical if and only if it is an equivariantization of a pointed fusion category.
\end{proof}

Now we are ready to describe the structure of an integral ASF modular category by equivariantizations.

\begin{cor}\label{cor36}
The modular category $\C$ is an equivariantization of a nilpotent fusion category of nilpotency class $2$.
\end{cor}
\begin{proof}
Let $\E\cong\Rep(G)$ be a proper maximal Tannakian subcategory of $\C$. Then we can form the de-equivariantization of $\C_G$ of $\C$ by $\E$. It is known that $\C_G$ has a faithful $G$-grading and the trivial component $(\C_G)_e$ is non-degenerate. By \cite[Proposition 4.56]{drinfeld2010braided}, $(\C_G)_e=(\E')_G$. Hence, $(\C_G)_e$ is pointed by Theorem \ref{thm34}, and hence $\C$ is an equivariantization of a nilpotent fusion category of nilpotency class $2$.
\end{proof}

Restricting to the case where $m=q$ is a prime number allows us to obtain one of the main results of this paper.

\begin{thm}\label{thm37}
Let $\C$ be an integral modular category of FP dimension $p^nq$, where $q<p$ are prime numbers. Then $\C$ is group-theoretical.
\end{thm}
\begin{proof}
It suffices to prove that $\FPdim(\C_{pt})$ can not be a power of $p$. Indeed, if $\FPdim(\C_{pt})$ is not a power of $p$ then $\FPdim(\C_{pt})=p^iq$ for some $2\leq i\leq n-2$ by Lemma \ref{lem30} and Lemma \ref{lem31}. By equalities (\ref{cptad}) and (\ref{fpdim}), $\FPdim(\C_{ad})=\FPdim((\C_{pt})')=p^{n-i}$. This means that $\C_{ad}$ is nilpotent \cite[Theorem 8.28]{etingof2005fusion}, and so is $\C$. Hence $\C$ is a group-theoretical fusion category by \cite[Corollary 6.2]{drinfeld2007group}.

Suppose on the contrary that $\FPdim(\C_{pt})=p^i$ for some $2\leq i\leq n-2$. Let $\E=\Rep(G)$ be a proper maximal Tannakian subcategory of $\C$. By Lemma \ref{newadded}, $\FPdim(\E)=p^j$ for some $j\leq i\leq n-2$. The M\"{u}ger centralizer $\E'$ of $\E$ in $\C$ is group-theoretical by Theorem \ref{thm34}. Under our assumption $q<p$, this fact implies that $\E'$ is nilpotent by \cite[Proposition 4.11]{DoTu2015}. The main result of \cite{drinfeld2007group} says that a braided nilpotent fusion category has a unique
decomposition as a tensor product of braided fusion categories whose FP dimensions are distinct prime powers. Again by equality (\ref{fpdim}),
we know $\FPdim(\E')=p^{n-j}q$. So we have
$$\E'\cong \A_q\boxtimes\A_{p^{n-j}},$$
where $\A_t$ is a fusion category with FP dimension $t$. By \cite[Corollary 8.30]{etingof2005fusion}, $\A_q$ is a pointed fusion category, hence $\E'$ and also $\C$ contain a pointed fusion subcategory of FP dimension $q$. This contradicts our assumption that $\FPdim(\C_{pt})$ is a power of $p$.
\end{proof}

\section{Low-dimensional integral modular categories}\label{sec6}
Let $\C$ be a fusion category, and let $1=a_0<a_1<\cdots<a_m$ be the distinct FP dimensions of simple objects of $\C$. If $n_i$ is the number of non-isomorphic simple objects of FP dimension $a_i$ then we say $\C$ has category type $(1, n_0; a_1, n_1; \cdots; a_m, n_m)$. We summarize some results in the literature to establish a criterion for an array $(1, n_0; a_1, n_1; \cdots; a_m, n_m)$ to be a category type of an odd-dimensional integral modular category.

\begin{lem}\label{lem41}
Let $(1, n_0; a_1, n_1; \cdots; a_m, n_m)$ be a category type of an integral modular category. Suppose that $\C$ has odd FP dimension. Then

(1) $a_i$ is odd and $n_i$ is even, for all $1\leqslant i\leqslant m$;

(2) $n_0$ divides $\FPdim(\C)$ and $n_ia_i^2$, for all $1\leqslant i\leqslant m$;

(3) $a_i^2$ divides $\FPdim(\C)$, for all $1\leqslant i\leqslant m$;

(4) $n_0a_m^2\leqslant \FPdim(\C)$.
\end{lem}

\begin{proof}
(1) If there exists $i$ such that $a_i$ is even, then $\FPdim(\C)$ is also even  by \cite[Lemma 5.3]{dong2014existence}. This is a contradiction. If there exists $i$ such that $n_i$ is odd then there must exist $X\in \Irr_{a_i}(\C)$ such that $X\cong X^\ast$. In this case $\FPdim(\C)$ is even by \cite[Lemma 5.2]{dong2014existence}. This is also a contradiction.

(2) It follows from \cite[Lemma 2.2]{dong2012frobenius}.

(3) It follows from \cite[Lemma 1.2]{etingof1998some}.

(4) Counting the FP dimension of every component of the universal grading of $\C$, we get $n_0a_m^2\leqslant \FPdim(\C)$ (Note that $n_0=|\mathcal{U}(\C)|$).
\end{proof}

Let $\C$ be a fusion category and let $G$ be a group acting on $\C$ by tensor autoequivalences. Recall from Section \ref{sec23} that we have the fusion category $\C^G$ of $G$-equivalent objects of $\C$. Let $F:\C^G\to \C$ be the forgetful functor. It is a surjective tensor functor. That is, for any object $Y$ of $\C$ there is an object $X$ in $\C^G$ such that $Y$ is a subobject of $F(X)$. The Lemma \ref{lem51} below is \cite[Lemma 7.2]{natale2014graphs}.

\begin{lem}\label{lem51}
Let $\C^G$ be the equivariantization of $\C$ under the action of $G$. Let $X$ be a simple object of $\C^G$. If $\FPdim(X)$ is relatively prime to the order of $G$ then $F(X)$ is a simple object of $\C$.
\end{lem}

\begin{thm}\label{thm52}
Let $\C$ be a modular category such that $\FPdim(\C)$ is odd and $\FPdim(\C)<1125$. Then $\C$ is group-theoretical.
\end{thm}

\begin{proof}
First, the assumption that $\FPdim(\C)$ is odd implies that $\C$ is integral \cite[Corollary 3.11]{gelaki2008nilpotent}. Second, the assumption that $\FPdim(\C)<1125$ implies that $\C$ is an integral ASF modular category except the case when  $\FPdim(\C)=675$. Suppose that $\C$ is an integral ASF modular category and the FP dimension $p^nm$ is odd and less than $1125$, where $p$ is a prime number, $m$ is a square-free natural number and ${\rm gcd}(p,m)=1$. If $m=1$ then $\FPdim(\C)=p^n$. In this case, $\C$ is group-theoretical by \cite[Corollary 6.8]{drinfeld2007group}. If $m\neq 1$ then $n\leq 5$. In this case, $\C$ is group-theoretical by Corollary \ref{cor32} and Corollary \ref{cor35}. Therefore, in the rest of the proof, we only need to consider the case when $\FPdim(\C)=675$.

By Lemma \ref{lem41}, the modular category $\C$ has the following possible types:
\begin{align*}
&(1,45;3,70),(1,27;3,72),(1,9;3,74),\\
&(1,3;3,8;5,6;15,2), (1,675).
\end{align*}

Suppose that $\C$ has category type $(1,45;3,70)$. Then the universal grading has $45$ components and every component has FP dimension $15$. Hence, any component contains at most $1$ simple object of FP dimension $3$, and so there should be at least $270$ non-isomorphic invertible simple objects. This is impossible since  $\FPdim(\C_{pt})=45$. The same argument shows that $\C$ can not have category types $(1,27;3,72)$ and $(1,9;3,74)$.

Suppose that $\C$ is of type $(1,3;3,8;5,6;15,2)$. It follows that every component $\C_g$ has FP dimension $225$ for all $g\in\mathcal{U}(\C)$, and the only possible category type of the trivial component $\C_{ad}$ is $(1,3;3,8;5,6)$. By \cite[Theorem 4.2]{dong2014existence}, $\C_{ad}$ has a non-trivial Tannakian subcategory $\D=\Rep(G)$. Counting the dimension of simple objects of $\C_{ad}$, we know that $\C_{ad}$ can not have fusion subcategory of FP dimension $15$ and $45$. Also, $\FPdim(\D)\ne 75$ by Lemma \ref{newadded}. So the only possibility of $\FPdim(\D)$ is $3$.

If $\FPdim(\D)=3$ then $\D=\C_{pt}$. By \cite[Corollary 6.8]{gelaki2008nilpotent}, $(\C_{ad})'=\D$. On the other hand, $\D$ is a fusion subcategory of $\C_{ad}$. Hence, $\D$ is the M\"{u}ger center of $\C_{ad}$. By \cite[Remark 2.3]{etingof2011weakly}, the de-equivariantization $(\C_{ad})_G$ of $\C_{ad}$ by $G$ is a modular category. Since $(\C_{ad})_G$ has FP dimension $75$, Corollary \ref{cor32} shows that $(\C_{ad})_G$ is pointed. But Lemma \ref{lem51} shows that $(\C_{ad})_G$ must contain simple objects of FP dimension $5$. This is a contradiction. Therefore, $\C$ must be pointed and we are done.
\end{proof}

\begin{rem}\label{rem52}
(1)\quad There exists non-group-theoretical integral modular categories whose FP dimension are even. Examples are given in \cite{bruillard2013classification}. In that paper, the authors studied non-group-theoretical $\mathbb{Z}_3$-graded $36$-dimensional modular categories and classified them, up to equivalence of fusion categories.

(2)\quad Natale studied low-dimensional modular categories in \cite{natale2013weakly}. She proved that if the FP dimension of a modular category is odd and less than $33075$ then this modular category is solvable.
\end{rem}


\section{Acknowledgements}
The author would like to thank Sonia Natale for her help with the proof of Theorem \ref{thm35}. The research of the author is supported by the Fundamental Research Funds for the Central Universities (KYZ201564), the Natural Science Foundation of China (11471282, 11201231) and the Qing Lan Project.




\end{document}